\documentclass[11pt]{article}

\usepackage[a4paper,margin=1in]{geometry}
\usepackage{amsmath,amssymb,amsthm,mathtools}
\usepackage{enumitem}
\usepackage{microtype}
\usepackage[colorlinks=true,linkcolor=blue,citecolor=blue,urlcolor=blue]{hyperref}

\newtheorem{theorem}{Theorem}[section]
\newtheorem{lemma}[theorem]{Lemma}
\newtheorem{proposition}[theorem]{Proposition}
\newtheorem{corollary}[theorem]{Corollary}

\newtheorem{question}[theorem]{Question}
\theoremstyle{definition}
\newtheorem{definition}[theorem]{Definition}

\newcommand{\F}{\mathbb F}
\newcommand{\calC}{\mathcal C}
\newcommand{\calE}{\mathcal E}
\newcommand{\calP}{\mathcal P}
\newcommand{\eps}{\varepsilon}

\title{Separating hash families with large universe II}
\author{Xin Wei\thanks{Extremal Combinatorics and Probability Group (ECOPRO),
Institute for Basic Science (IBS), Daejeon, South Korea.
Email: \texttt{weixinma@ibs.re.kr}.
Supported by the Institute for Basic Science (IBS-R029-C4).
His work on this project was supported by the Excellent Doctoral Students
Study Abroad Support Program of the University of Science and Technology
of China during his visit to the National University of Singapore.}
\and Gennian Ge\thanks{School of Mathematical Sciences, Capital Normal
University, Beijing 100048, China. Email: \texttt{gnge@zju.edu.cn}.
Supported by the National Key Research and Development Program of China
under Grant 2025YFC3409900, the National Natural Science Foundation of
China under Grant 12231014, and the Beijing Scholars Program.}
}
\date{}

\begin{document}

\maketitle

\begingroup
\renewcommand{\thefootnote}{}
\footnotetext{This work was supported by NSFC
under Grant~12671395.}
\endgroup

\begin{abstract}
For fixed integers $t,k\ge 2$, put $N=(t-1)(k-1)+1$.  
We prove that the maximum universe size of a $t$-perfect hash family of length $N$ over an alphabet of size $q$ satisfies
  $ p_t(N,q)\ge q^{k-o(1)}$.
Together with monotonicity in the length, this gives a positive answer to a question of Blackburn, Etzion, Stinson and Zaverucha for every choice of feasible parameters.  
Our proof combines the Nash--Williams--Tutte theorem, a graph-theoretic elimination argument, and a solution-free-set theorem established in our previous work.
\end{abstract}

\section{Introduction}

Let $X$ and $Y$ be sets of sizes $n$ and $q$, respectively.  
An $(N;n,q)$-hash family is an indexed family of $N$ functions from $X$ to $Y$. Here we call $X$ and $Y$ the {\it universe} and {\it alphabet} of the hash family, respectively.
Equivalently, an $(N; n, q)$-hash family is an $N\times n$ array over an alphabet of size $q$, whose columns are indexed by $X$.  
Throughout the paper, we consider only hash families with pairwise distinct columns.
We further call it a $t$-perfect hash family, abbreviated $t$-PHF, if every set of $t$ columns is separated by some row; that is, its $t$ entries in that row are pairwise distinct.  
We use $\operatorname{PHF}(N;n,q,t)$ to denote a $(N; n,q)$-hash family which is a $t$-PHF, and write $p_t(N,q)$ for the largest possible number of columns.

More generally, let $s\ge2$ and let $w_1,\ldots,w_s$ be positive integers.
We say that a hash family is of \emph{separating type}
$\{w_1,\ldots,w_s\}$ if, for every collection of pairwise disjoint column
sets $C_1,\ldots,C_s$ with $|C_i|=w_i$, there is a row that separates
$C_1,\ldots,C_s$.  Here, a row separates these column sets means the sets of
symbols appearing in them are pairwise disjoint.  The type is a multiset,
so the weights $w_i$ need not be distinct.
An $(N;n,q)$-hash family of separating type $\{w_1,\ldots,w_s\}$ is called
a \emph{separating hash family}, denoted by
$\operatorname{SHF}(N;n,q,\{w_1,\ldots,w_s\})$.  Write
$C(N,q,\{w_1,\ldots,w_s\})$ for the largest possible number of columns
of such a family.  When $s=t$ and $w_1=\cdots=w_s=1$, an $\operatorname{SHF}(N;n,q,\{1,\ldots,1\})$ is exactly a $\operatorname{PHF}(N;n,q,t)$, and $C(N,q,\{1,\ldots,1\})=p_t(N, q)$.

Perfect hash families have a long history in algorithms and combinatorics;
see Mehlhorn~\cite{mehlhorn2013data} and the constructions in
\cite{alon1996derandomization,blackburn2000perfect,stinson2000new,walker2007perfect}.
Separating hash families were introduced as a common generalization of
perfect hash families, frameproof codes and related structures
\cite{stinson2008generalized}.  Bounds on their maximum universe size have
been studied in
\cite{blackburn2008bound,bazrafshan2011bounds,ge2019some,shangguan2016separating}.

The following classical bounds describe the gap that we address.  The upper
bound is due to Blackburn, Etzion, Stinson and
Zaverucha~\cite[Theorem~1]{blackburn2008bound}.  The lower bound follows from
Blackburn's probabilistic construction~\cite{blackburn2000perfect}.
The explicit form below is also recorded in \cite[Section~1]{WEI2025106075}.

\begin{theorem}[Classical bounds]\label{thm:classical-bounds}
Fix integers $N\ge1$, $s\ge2$ and $1\le w_1\le\cdots\le w_s$, and set
$u=\sum_{i=1}^s w_i$.  For all sufficiently large integers $q$,
\begin{equation*}
    \frac{1}{2u}\left(\frac{q}{\binom{u}{2}}\right)^{N/(u-1)}\le C(N,q,\{w_1,\ldots,w_s\})
 \le (w_1w_2+u-w_1-w_2)q^{\lceil N/(u-1)\rceil}.
\end{equation*}
\end{theorem}

When $(u-1)\mid N$, these bounds give
$C(N,q,\{w_1,\ldots,w_s\})=\Theta(q^{N/(u-1)})$ for fixed parameters.
For the remaining lengths, Blackburn et al.\ \cite{blackburn2008bound} posed the following question.

\begin{question}[Blackburn, Etzion, Stinson and Zaverucha
 {\cite[Section~5]{blackburn2008bound}}]\label{ques:blackburn}
Fix integers $N\ge1$, $s\ge2$ and $w_1,\ldots,w_s\ge1$, and set
$u=\sum_{i=1}^s w_i$.  Suppose that $(u-1)\nmid N$.  For every
$\eps>0$, are there infinitely many integers $q$ for which an
$\operatorname{SHF}(N;n,q,\{w_1,\ldots,w_s\})$ exists with
\[
                 n\ge q^{\lceil N/(u-1)\rceil-\eps}?
\]
\end{question}

The first nontrivial case is exponent two.  For $t=3,4$, Shangguan and
Ge~\cite{shangguan2016separating} constructed $t$-PHFs of length $t$ and
size $q^{2-o(1)}$.  In \cite{WEI2025106075}, Wei, Zhang and Ge
proved the same bound for every fixed $t\ge3$.
Consequently, Question~\ref{ques:blackburn} was answered for
$N\in[u,2u-3]$.  Our main
result obtains every exponent.

\begin{theorem}\label{thm:main}
Let $t,k\ge2$ be fixed integers and let $N=(t-1)(k-1)+1$.  Then, as $q\to\infty$,
\[
       p_t(N,q)\ge q^{k-o(1)}.
\]
\end{theorem}

The case $t=2$ is immediate, since all $q^k$ distinct words of length $k$
form a $2$-PHF.
For $t\ge3$, the case $k=2$ recovers the result of \cite{WEI2025106075}. Our result treats it together with all larger $k$.
Theorem~\ref{thm:main} gives the following confirmative answer to Question~\ref{ques:blackburn}.  

\begin{corollary}\label{cor:full}
Let $s\ge2$ and $N\ge1$ be fixed integers, let $w_1,\ldots,w_s$ be fixed positive integers, set $u=\sum_{i=1}^s w_i$, and suppose $(u-1)\nmid N$.
Then, as $q\to\infty$,
\[
 C(N,q,\{w_1,\ldots,w_s\})
 \ge q^{\lceil N/(u-1)\rceil-o(1)}.
\]
\end{corollary}

\begin{proof}
Let $k=\lceil N/(u-1)\rceil$.  If $k=1$, the $q$ constant words
$(a,\ldots,a)$, one for each alphabet symbol $a$, form a $u$-PHF for
$q\ge u$ and give the asserted bound.  If $k\ge2$, then
\[
                   N\ge(u-1)(k-1)+1.
\]
Apply Theorem~\ref{thm:main} with $t=u$ and append arbitrary rows to the
resulting family to reach length $N$.  The number of columns remains
$q^{k-o(1)}$.  
From the perfect-hash property,
a $u$-PHF separates every type whose total weight is $u$. This proves the corollary.
\end{proof}

The divisible case is covered by Theorem~\ref{thm:classical-bounds}.
Together, these results determine the exponent of $C(N,q,\{w_1,\ldots,w_s\})$ for every fixed positive length and type.

\subsection{Main ideas}

Our construction is motivated by the hypergraph construction of Alon and
Shapira~\cite{alon2006extremal}. Viewing columns as codewords identifies
the universe size with the size of the code. Fix integers $t\ge3$, $k\ge2$,
and put $N=(t-1)(k-1)+1$. We work over $\F_p$ for a sufficiently large prime $p$ with $q/2<p\le q$. We choose an $N\times k$ matrix $M=M(p)$ of column rank $k$
over $\F_p$, so that $\boldsymbol z\mapsto M\boldsymbol z$ is injective
(Corollary~\ref{cor:minors}). Splitting the parameter column vector into a
\emph{head} $x\in\F_p$ and a \emph{tail} $\boldsymbol y\in\F_p^{k-1}$,
we consider the code (Section~\ref{sec:code})
\[
 \calC_p=\left\{M\begin{pmatrix}x\\ \boldsymbol y\end{pmatrix}:
 x\in S,\ \boldsymbol y\in\F_p^{k-1}\right\}\subseteq\F_p^N,
\]
where $S=S(p)\subseteq\F_p$ will be chosen below. The tail ranges over
the whole space, so $|S|=p^{1-o(1)}$ gives
$|\calC_p|=|S|p^{k-1}=p^{k-o(1)}$.

The proof proceeds by contradiction.
If $t$ distinct codewords are not separated, each row supplies an equality
between two of them. We seek suitable collision equations involving a
subset of these codewords. Our aim is to eliminate all tail parameters,
obtain an equation in the heads, and use the solution-free property of
$S$ to force these heads to coincide. The original collision equations
must then force the tails to coincide as well, contradicting distinctness.
This argument has two main difficulties.

The first difficulty is effective elimination. For $r$ codewords,
subtracting a reference tail leaves $(r-1)(k-1)$ scalar unknowns.
One more collision equation provides enough equations, but the
coefficient matrix must have full column rank. We record collisions
by a multigraph with distinct row labels and select a minimal
configuration with more equations than tail differences.
The Nash--Williams--Tutte theorem
\cite{nashwilliams1961edge,tutte1961decomposing} supplies spanning-tree
decompositions that certify the relevant determinant polynomials are
nonzero (Lemma~\ref{lem:Q}). The hierarchical entries of $M$
ensure that their integer values are nonzero and remain so modulo $p$
(Lemma~\ref{lem:dominance}). The same rank condition forces all tail
differences to vanish once the heads coincide.

The second difficulty is controlling the resulting head equation.
Its coefficients sum to zero, but applying the solution-free-set theorem
from our previous work~\cite{WEI2025106075} requires more: after a
permutation, these coefficients must be cyclic differences of a plastic
set, whose elements satisfy specified logarithmic separation conditions
(Theorem~\ref{thm:plastic}). The hierarchy in $M$ also separates
the edge contributions to the head equation (equation~\eqref{eq:flow-order}).
Graph contractions and their reverse expansions then order its
coefficients so that their partial sums, after translation, form
such a plastic set (Section~\ref{sec:head}).

Finally, $S$ must work for every collision pattern. These patterns have
a fixed finite indexing set depending only on $t,k$. The plastic-set
theorem and the common solution-free-set lemma
(Lemma~\ref{lem:common}) yield one set $S$ of size $p^{1-o(1)}$;
restricting the integer interval preserves its solution-free property
over $\F_p$ (Section~\ref{sec:common-head}). Every head equation then
has only solutions with all coordinates equal in $S$. A failed separation
therefore forces equality first of the heads and then of the tails,
completing the contradiction.

\medskip
\noindent\textbf{Organization and notation.}
We first establish the properties needed by the construction, then combine
them to verify the perfect-hash property.
Section~\ref{sec:solution-free} recalls
the solution-free theorem.  Section~\ref{sec:matrix} constructs the matrix,
and Section~\ref{sec:error} proves the elimination determinant property.
Section~\ref{sec:head} treats the plastic partial sums, and
Section~\ref{sec:common-head} constructs the common head set.
Section~\ref{sec:code} proves the perfect-hash property and completes the
proof of Theorem~\ref{thm:main}.

We write $[n]=\{1,\ldots,n\}$ for a positive integer $n$, and all
logarithms below are natural.  Unless stated
otherwise, asymptotic constants may depend on the fixed integers $t,k$ and
the fixed hierarchy parameter $c$, but not on $p$ or $q$.

\section{Plastic sets and invariant equations}
\label{sec:solution-free}

We first recall the equations and solution-free sets used in the
construction.  All equations in this section are over the integers.

\begin{definition}[Invariant equations and solution-free sets]
\label{def:invariant}
Let $r\ge2$ and let $a_1,\ldots,a_r$ be nonzero integers.  A homogeneous
linear equation
\[
                  L:\quad a_1x_1+\cdots+a_rx_r=0
\]
is called \emph{invariant} if $a_1+\cdots+a_r=0$.  A solution is
\emph{trivial} if $x_1=\cdots=x_r$.  An integer set $A$ is
\emph{$L$-free} if every solution in $A^r$ is trivial.  For a family of
equations, a \emph{common solution-free set} is a set that is $L$-free for
each equation $L$ in the family.
\end{definition}

The term invariant refers to simultaneous translation of the variables:
for every integer $d$,
\[
       \sum_{i=1}^r a_i(x_i+d)=\sum_{i=1}^r a_ix_i.
\]
Thus a tuple is a solution if and only if its translate by $(d,\ldots,d)$
is a solution, and translation also preserves triviality.  In particular,
if $A$ is $L$-free, then $A+d=\{x+d:x\in A\}$ is $L$-free.

\begin{definition}[Permutation equations]\label{def:permutation-equation}
For a sequence $U=(u_1,\ldots,u_r)$ of distinct integers, with $r\ge2$, define
\begin{equation}\label{eq:LU}
 L(U):\quad (u_1-u_r)x_1+(u_2-u_1)x_2+\cdots+(u_r-u_{r-1})x_r=0.
\end{equation}
Its coefficients are nonzero and sum to zero, so $L(U)$ is invariant.
\end{definition}

Let $h$ be a positive integer tending to infinity.  It will be the length
of the interval in which we seek a solution-free set.  The number $r$ of
elements is fixed, while $R(h)=\{b_1(h)<\cdots<b_r(h)\}$ may vary with $h$.
We suppress the argument $h$ when writing its elements.  In
Section~\ref{sec:common-head}, we will apply the following results with
$h=m_0(p)$, an interval length chosen from the coefficients of our equations.

\begin{definition}\label{def:plastic}
Fix an integer $r\ge2$ and real numbers $0<\beta<\alpha<1$.  A family of sets
$R(h)=\{b_1(h)<\cdots<b_r(h)\}$ of positive integers is
$(\alpha,\beta)$-plastic with respect to $h$ if, as $h\to\infty$,
\[
 \log b_1=\Omega((\log h)^\beta),\qquad
 \log b_r=O((\log h)^\alpha),
\]
and $\log b_i=O((\log b_{i+1})^\alpha)$ for every $i\in[r-1]$.
The constants in these estimates are independent of $h$.
\end{definition}
The same terminology applies when the elements and interval length are
indexed by $p$, with $h=h(p)\to\infty$; the displayed estimates are then
understood as $p\to\infty$.

Notice that the growth conditions in Definition~\ref{def:plastic}
also explain the ordering of the elements.  Indeed, for each $i<r$,
\[
 \frac{\log b_i}{\log b_{i+1}}
 =O\bigl((\log b_{i+1})^{\alpha-1}\bigr)=o(1),
\]
since all $b_i$ tend to infinity.  Thus $b_i=b_{i+1}^{o(1)}<b_{i+1}$
for all sufficiently large $h$.  The threshold may depend on the
constants in the growth conditions.  We will also use the elementary
estimate
\[
 \exp\bigl(O((\log h)^\theta)\bigr)
 =h^{O((\log h)^{\theta-1})}=h^{o(1)}
 \qquad(0<\theta<1).
\]

The next theorem is the form of \cite[Theorem~2.1]{WEI2025106075} that we
need.  We use the same set size for $R(h)$ and $U$; the cited result also
allows sequences chosen from proper subsets of $R(h)$.

\begin{theorem}[Plastic permutation equations]\label{thm:plastic}
Fix $r\ge3$, $0<\alpha<1$ and $0<\beta\le\alpha^r$, and let
$R(h)\subseteq[h]$ be an $(\alpha,\beta)$-plastic $r$-set with respect to
$h$.  There are constants $0<\gamma=\gamma(r,\alpha)<1$, $K>0$ and $h_0$
such that, for every integer $h\ge h_0$ and every permutation $U$ of
$R(h)$, there is an $L(U)$-free set
$A_U\subseteq\{0,1,\ldots,h-1\}$ satisfying
\[
               |A_U|\ge h\exp\bigl(-K(\log h)^\gamma\bigr).
\]
The constants $K$ and $h_0$ may depend on the fixed parameters and the
plastic family, but not on $h$ or on the chosen permutation.
\end{theorem}

For clarity, the permutations are indexed by the $r!$ permutations $\pi$
of $[r]$, with $U_\pi(h)=(b_{\pi(1)}(h),\ldots,b_{\pi(r)}(h))$.
Taking the maximum of the constants and thresholds for these finitely many
choices gives one $K$ and one $h_0$ valid for all of them.  The sets $A_U$
may still differ.  Since $\gamma<1$, their common lower bound is
$h^{1-o(1)}$.  The cited theorem gives sets in $[h]$; translating by $-1$
gives the interval above by Definition~\ref{def:invariant}.
Changing the base of the logarithm from $2$ to $e$ only changes the
constants in these estimates.

We shall also need to avoid a fixed finite family of equations
simultaneously.  We use the following result from
\cite{Ge2021SparseHN}, recalled in \cite[Lemma~2.4]{WEI2025106075}.

\begin{lemma}[Common solution-free set]\label{lem:common}
Fix $0<\gamma<1$, $K>0$ and a positive integer $\ell$.  For each sufficiently
large integer $h$, let $L_1,\ldots,L_\ell$ be invariant linear equations
with integer coefficients, which may depend on $h$.  Suppose that for each
$i\in[\ell]$ there is an $L_i$-free set $A_i\subseteq[h]$ satisfying
\[
                   |A_i|\ge h\exp\bigl(-K(\log h)^\gamma\bigr).
\]
Then there is a constant $K'>0$, depending only on $\ell,\gamma,K$, and a
set $A\subseteq[h]$ which is $L_i$-free for every $i\in[\ell]$ and satisfies
\[
                    |A|\ge h\exp\bigl(-K'(\log h)^\gamma\bigr).
\]
\end{lemma}

\begin{proof}
For $i=2,\ldots,\ell$, choose $d_i$ independently and uniformly from
$\{-(h-1),\ldots,h-1\}$.  Averaging the size of
$A_1\cap(A_2+d_2)\cap\cdots\cap(A_\ell+d_\ell)$ gives
\[
 \frac{|A_1|\cdots|A_\ell|}{(2h-1)^{\ell-1}}
 \ge \frac{h}{2^{\ell-1}}
       \exp\bigl(-\ell K(\log h)^\gamma\bigr).
\]
Indeed, for each $x\in A_1$ and each $i\ge2$, precisely $|A_i|$ choices
of $d_i$ place $x$ in $A_i+d_i$.  Some choice of the translations therefore
attains at least the displayed average.  The resulting intersection lies
in $[h]$ and is free of every $L_i$, since translations preserve the
solution-free property.  Absorbing $2^{\ell-1}$ into the constant in the
exponent proves the claim.  For $\ell=1$, simply take $A=A_1$.
\end{proof}

By the same translation property, the lemma holds with $[h]$ replaced by
$\{0,\ldots,h-1\}$.

\section{The hierarchical matrix}
\label{sec:matrix}

We first define the integer sequences that will provide the entries of our
matrix.  Throughout the construction, $p$ tends to infinity through the
primes; reduction modulo $p$ will be stated explicitly when it is used.

\begin{definition}[Hierarchical sequences]\label{def:hierarchy}
Fix a positive integer $L$ and a real number $0<c<1$.  For each prime $p$
and each $j\in[L]$, define
\begin{equation*}
        m_j(p)=\left\lfloor\exp\bigl((\log p)^{c^j}\bigr)\right\rfloor.
\end{equation*}
We write $\mathbf m(p)=(m_1(p),\ldots,m_L(p))$, and omit the argument $p$
when no confusion can arise.
\end{definition}

\begin{definition}[The hierarchical matrix]\label{def:matrix}
For the fixed integers $t\ge3$ and $k\ge2$, put $L=Nk$ and fix $0<c<1$.
Using the sequence in Definition~\ref{def:hierarchy}, define the
$N\times k$ integer matrix $M(p)$ by
\[
 M_{i,j}(p)=m_{(j-1)N+i}(p)\qquad(i\in[N],\ j\in[k]).
\]
We write $M=M(p)$ when $p$ is understood.  Thus the entries are numbered
from top to bottom in the first column, then from top to bottom in the
second column, and so on.
\end{definition}

Since $L$ is fixed,
$\log m_j=(1+o(1))(\log p)^{c^j}$ uniformly for $j\in[L]$.
For any fixed positive integer $D$ and $j<L$, we have
\[
 \log\frac{m_j}{m_{j+1}^D}
 =(1+o(1))(\log p)^{c^j}
   -D(1+o(1))(\log p)^{c^{j+1}}\longrightarrow\infty.
\]
Thus a fixed power of a later entry is smaller
than even the first power of an earlier entry.  This separation implies
that polynomials of bounded degree and bounded integer coefficients cannot
vanish after substitution, once $p$ is sufficiently large.  The following
lemma states this property with a threshold valid for all such polynomials.

\begin{lemma}[Dominant monomial]\label{lem:dominance}
Fix a positive integer $L$, a real number $0<c<1$, and positive integers
$D,H$.  Let $m_1(p),\ldots,m_L(p)$ be as in
Definition~\ref{def:hierarchy}, and let $X_1,\ldots,X_L$ be indeterminates.
There exists $p_0=p_0(L,c,D,H)$ such that, for every prime $p\ge p_0$
and every nonzero polynomial $P\in\mathbb Z[X_1,\ldots,X_L]$ of total
degree at most $D$ whose coefficients have absolute value at most $H$,
\[
                P(\mathbf m(p)):= P\bigl(m_1(p),\ldots,m_L(p)\bigr)\ne0.
\]
Moreover, as $p\to\infty$ through the primes, $\bigl|P(\mathbf m(p))\bigr|=p^{o(1)}$,
uniformly over all such $P$.  If $P$ is nonconstant, then $|P(\mathbf m(p))|\to\infty$.
\end{lemma}

\begin{proof}
For any exponent vector $\mathbf e=(e_1,\ldots,e_L)\in\mathbb Z_{\ge0}^L$, write
 $|\mathbf e|=\sum_{j=1}^L e_j$,
 $\mathbf X^{\mathbf e}=\prod_{j=1}^L X_j^{e_j}$, and
 $\mathbf m^{\mathbf e}=\prod_{j=1}^L m_j^{e_j}$.
Expand $P$ as
\[
              P(\mathbf X)=\sum_{|\mathbf e|\le D}
                              c_{\mathbf e}\mathbf X^{\mathbf e}.
\]
Order the exponent vectors lexicographically. That is, we say $\mathbf e>\mathbf f$ if, at the first coordinate $s$ where they differ, $e_s>f_s$.
Let $\mathbf e$ be the largest exponent vector such that $c_{\mathbf e}\ne0$. For any other vector
$\mathbf f$ with $c_{\mathbf f}\ne0$, let $s$ be the first coordinate such that $e_s\ne f_s$.
As all exponents are at most $D$, we have
\[
 \log\frac{\mathbf m^{\mathbf e}}{\mathbf m^{\mathbf f}}
 \ge \log m_s-D\sum_{j>s}\log m_j
 =(1-o(1))(\log p)^{c^s}.
\]
The ratio $\mathbf m^{\mathbf f}/\mathbf m^{\mathbf e}$ therefore tends to
zero.  There are at most $\binom{L+D}{D}$ terms, each coefficient has
absolute value at most $H$, and $|c_{\mathbf e}|\ge1$.  Consequently,
\[
 \left|\frac{P(\mathbf m)}{c_{\mathbf e}\mathbf m^{\mathbf e}}-1\right|
 \le H\sum_{\substack{\mathbf f\ne\mathbf e\\c_{\mathbf f}\ne0}}
          \frac{\mathbf m^{\mathbf f}}{\mathbf m^{\mathbf e}}
 =o(1).
\]
Only finitely many exponent vectors occur for fixed $L,D$, so this estimate
is uniform over all the polynomials in the statement. Choose $p_0$ such
that $|P(\mathbf m)/(c_{\mathbf e}\mathbf m^{\mathbf e})-1|\le1/2$
for every prime $p\ge p_0$. This proves nonvanishing. Also,
\[
 |P(\mathbf m)|\le H\binom{L+D}{D}m_1^D
 =\exp\bigl(O((\log p)^c)\bigr)=p^{o(1)}.
\]
If $P$ is nonconstant, its leading exponent vector is nonzero, so
$\mathbf m^{\mathbf e}\to\infty$ and hence $|P(\mathbf m)|\to\infty$.
\end{proof}

\begin{corollary}\label{cor:minors}
Let $M(p)$ be the matrix in Definition~\ref{def:matrix}.
There exists a positive integer $p_0$, depending only on $t,k,c$, such that for every prime $p\ge p_0$, the determinant of every square submatrix of $M(p)$ is a nonzero integer of absolute value less than $p$, and hence remains nonzero modulo $p$.

In particular, every submatrix formed by $k$ rows and all $k$ columns is invertible over $\mathbb F_p$ after reduction modulo $p$.
\end{corollary}

\begin{proof}
Replace each matrix entry $m_j$ by the corresponding indeterminate $X_j$.
Note that for a square submatrix of order $d$, the determinant is a nonzero integer
polynomial of degree $d$ with coefficients $\pm1$. That is because different permutations select different sets of entries and hence different monomials.
Since $d\le k$, Lemma~\ref{lem:dominance} with $L=Nk$, $D=k$ and $H=1$ shows that every
such determinant is a nonzero integer of absolute value $p^{o(1)}<p$ for
sufficiently large $p$.  Its reduction modulo $p$ is therefore nonzero.
\end{proof}

\section{Exact collision multigraphs}
\label{sec:error}

All subsequent claims concern sufficiently large primes $p$, with
thresholds depending only on the fixed construction parameters $t,k,c$.

By regarding each column as a codeword, an $(N;n,q)$-hash family
can be viewed as an indexed collection of $q$-ary words of length $N$.
We say a \emph{collision} happens between two codewords in row $i$ if their
entries in row $i$ are equal.  Suppose that $t$ distinct codewords are not
separated by any row.  We record their collisions by a multigraph:
\begin{itemize}[nosep]
\item Each vertex represents one of the $t$ codewords.
\item In each row $i$, choose one pair that agrees in that row and join
      the corresponding vertices by an edge labeled $i$.
\end{itemize}
Thus each row contributes one edge.  Different rows may give parallel
edges, but no edge is a loop.  When we pass to a smaller vertex set or
select some of the edges, each row contributes at most one edge.
We first study these graphs independently of the code.  Throughout the
rest of the paper, a graph means a finite loopless multigraph, and all
edge counts include multiplicities.

For a multigraph $G$, let $e(G)$ denote its number of edges and let
$e_G(X)$ denote the number of edges with both endpoints in $X\subseteq V(G)$.  For $r\ge2$, put
\[
                 N_r=(k-1)(r-1)+1.
\]
\begin{definition}[Exact collision graphs]\label{def:exact}
An $r$-vertex loopless multigraph $G$, with $2\le r\le t$, is
\emph{exact} if
\begin{equation}\label{eq:exact}
 e(G)=N_r
 \quad\text{and}\quad
 e_G(X)\le(k-1)(|X|-1)
\end{equation}
for every nonempty proper subset $X\subsetneq V(G)$.  An
\emph{exact labeled graph} is an exact graph equipped with an injective
map $\rho:E(G)\to[N]$.  We call $\rho(e)$ the row label of the edge $e$.
In particular, parallel edges receive different labels.
\end{definition}

The inequality in \eqref{eq:exact} is automatic for a singleton, since the graph has no loops.
Every exact graph is connected. In fact, if it had $a\ge2$ components,
summing that inequality over their vertex sets would give
$e(G)\le(k-1)(r-a)<N_r$.

Fix an exact labeled graph $G$ as in Definition~\ref{def:exact}.
Choose an initial orientation of its
edges and a vertex $o\in V(G)$ as the root.  Fix an ordering of the vertices;
when $V(G)=[r]$, use the natural order.  The first column of $M$ gives the
head coefficients, while columns $2,\ldots,k$ give the tail coefficients.
For each edge $e$ and each $j\in\{2,\ldots,k\}$, write
$m_{e,j}=M_{\rho(e),j}(p)$ for the corresponding tail entry.

\begin{definition}[Tail-difference matrix]\label{def:tail-matrix}
Let $X_{e,j}$, for $e\in E(G)$ and $2\le j\le k$, be distinct
indeterminates, and let $\mathbf X$ denote their collection.
Define the polynomial matrix $\mathsf Q_o(\mathbf X)$
with rows indexed by
\[
          (v,j)\in(V(G)\setminus\{o\})\times\{2,\ldots,k\},
\]
and its columns by $E(G)$.  If $e$ is oriented from $u$ to $v$, then
\[
 (\mathsf Q_o(\mathbf X))_{(w,j),e}=
 \begin{cases}
 X_{e,j},&w=u,\\
 -X_{e,j},&w=v,\\
 0,&\text{otherwise},
 \end{cases}
\]
where the rows belonging to $o$ are omitted.  The integer matrix
$Q_o=Q_o(G;p)$ is obtained by substituting $X_{e,j}=m_{e,j}$.
For a matrix whose columns are indexed by $E(G)$, a subscript
$\widehat e$ means that column $e$ is deleted.
\end{definition}

Both matrices have $D_r=(k-1)(r-1)$ rows and $D_r+1=N_r$ columns.
We order the rows first by $j$, then by the fixed vertex order, and
the columns by increasing row label.

For a vector $Y$ indexed by $(v,j)\in (V(G)\backslash\{o\})\times \{2,\ldots,k\}$ and  $e\in E(G)$ directed from $u$ to $v$, by setting $Y_{(o,j)}=0$ for any $j\in \{2,\ldots,k\}$,
the coordinate of $Q_o^{\mathsf T}Y$ indexed by $e$ is exactly
\begin{equation}\label{eq:Q0Y}
     (Q_o^{\mathsf T}Y)_e=\sum_{j=2}^k m_{e,j}(Y_{(u,j)}-Y_{(v,j)}).
\end{equation}

Thus $Q_o^{\mathsf T}$ is the coefficient matrix for the collision
equations in the tail differences from the root.

We will use the following theorem to decompose $G-e_0$ into spanning
trees. This decomposition will identify a monomial with nonzero
coefficient in each elimination determinant.

\begin{theorem}[Nash--Williams--Tutte {\cite{nashwilliams1961edge,tutte1961decomposing}}]\label{thm:tree-packing}
Let $H$ be a finite loopless multigraph with at least two vertices,
and let $d$ be a positive integer.
Then $H$ contains $d$ pairwise edge-disjoint spanning trees if and only if,
for every partition $\calP$ of $V(H)$ into nonempty parts,
\[
                 e_H(\calP)\ge d(|\calP|-1),
\]
where $e_H(\calP)$ is the number of edges with endpoints in different
parts, counted with multiplicity.
\end{theorem}

\begin{lemma}[Non-vanishing elimination determinants]\label{lem:Q}
Let $G$ be an exact labeled graph on $r$ vertices with a fixed orientation and root
$o$.  For every $e_0\in E(G)$,
$\det\bigl((\mathsf Q_o(\mathbf X))_{\widehat e_0}\bigr)$ is a nonzero polynomial
in the indeterminates $X_{e,j}$.  There exists $p_0=p_0(t,k,c)$ such that, for every prime $p\ge p_0$
and all these choices of $G,o,e_0$ and orientation, its integer value $\det\bigl((Q_o)_{\widehat e_0}\bigr)$ is nonzero and has absolute value
less than $p$.  In particular, $Q_o$ has row rank $D_r$ over $\F_p$.
\end{lemma}

\begin{proof}
Let $H=G-e_0$.  Then $e(H)=(k-1)(r-1)$.
For every partition $\calP$ of $V(G)$ with at least two parts, as $G$ is exact,
\[ e_H(\calP) =e(H)-\sum_{X\in\calP}e_H(X)
 \ge(k-1)(r-1)-(k-1)\sum_{X\in\calP}(|X|-1)
 =(k-1)(|\calP|-1).\]

The partition with one part satisfies the required inequality as well.
Theorem~\ref{thm:tree-packing} therefore gives $k-1$ edge-disjoint
spanning trees $T_2,\ldots,T_k$ in $H$, which contain exactly $(k-1)(r-1)=e(H)$ edges. Hence, this also yields an edge partition
\begin{equation}\label{eq:tree-partition}
    E(H)=E(T_2)\mathbin{\dot\cup}\cdots  \mathbin{\dot\cup}E(T_k).
\end{equation}

Let
$A=(\mathsf Q_o(\mathbf X))_{\widehat e_0}$
and, for \(j\in\{2,\ldots,k\}\), let
$R_j=(V(G)\setminus\{o\})\times\{j\}$.
Thus \(R_2,\ldots,R_k\) partition the rows of \(A\), while the columns of
\(A\) are indexed by \(E(H)\), and \(|R_j|=r-1\) for every \(j\).

We first group the permutation expansion of \(\det A\) according to these
row blocks.  With the fixed row and column orders, regard a permutation
term as a bijection
\[
\delta:\bigcup_{j=2}^k R_j\longrightarrow E(H).
\]
For such a bijection, define
$F_j=\delta(R_j)$ for any $j\in\{2,\ldots,k\}$.
Then \((F_2,\ldots,F_k)\) is an ordered partition of \(E(H)\) with
\(|F_j|=r-1\) for every \(j\). Conversely, we say that a permutation term
satisfies an ordered partition \((F_2,\ldots,F_k)\) if
\(\delta(R_j)=F_j\) for every \(j\).

If such a permutation term is nonzero, then every entry selected from
column \(e\in F_j\) and row block \(R_j\) is either \(X_{e,j}\) or
\(-X_{e,j}\). Hence every nonzero term satisfying
\((F_2,\ldots,F_k)\) has the same monomial
\[
\prod_{j=2}^k\prod_{e\in F_j}X_{e,j}.
\]
Collecting the permutation terms according to their ordered partitions,
we may therefore write
\[
\det A
=
\sum_{(F_2,\ldots,F_k)}
c(F_2,\ldots,F_k)
\prod_{j=2}^k\prod_{e\in F_j}X_{e,j},
\]
where the sum is over all ordered partitions of \(E(H)\) with
\(|F_j|=r-1\), and \(c(F_2,\ldots,F_k)\in\mathbb Z\).
Since the variables \(X_{e,j}\) are distinct for different pairs
\((e,j)\), different ordered partitions give different monomials.

We now consider the particular ordered partition
$F_j=E(T_j)$, $j\in\{2,\ldots,k\}$
given by \eqref{eq:tree-partition}.
For each \(j\), let
$A_j=A[R_j,E(T_j)]$
be the corresponding \((r-1)\times(r-1)\) submatrix of $A$.  Note that
\[
A_j
=
B_j\operatorname{diag}\bigl(X_{e,j}:e\in E(T_j)\bigr),
\]
where $\operatorname{diag}\bigl(X_{e,j}:e\in E(T_j)\bigr)$ means the diagonal matrix and \(B_j\) is the signed incidence
matrix of \(T_j\) with the row corresponding to the root \(o\) deleted.

The permutation terms producing the monomial
$\mathfrak m(\mathbf X)
=
\prod_{j=2}^k\prod_{e\in E(T_j)}X_{e,j}$
are precisely those satisfying the ordered partition
\((E(T_2),\ldots,E(T_k))\).  After permuting the columns so that the
sets \(E(T_2),\ldots,E(T_k)\) occur consecutively, their total
contribution is
$\varepsilon\prod_{j=2}^k\det A_j$
for some fixed \(\varepsilon\in\{1,-1\}\).  Since \(T_j\) is a spanning
tree, the determinant of its reduced signed incidence matrix is
\(\pm1\). Hence
\[
\det A_j
=
\det B_j\prod_{e\in E(T_j)}X_{e,j}
=
\pm\prod_{e\in E(T_j)}X_{e,j}.
\]
It follows that
\[
c(E(T_2),\ldots,E(T_k))\,\mathfrak m(\mathbf X)
=
\varepsilon\prod_{j=2}^k\det A_j
=
\pm\mathfrak m(\mathbf X),
\]
and therefore
$c(E(T_2),\ldots,E(T_k))=\pm1$.
Thus \(\det A\) is not the zero polynomial.

Finally, since the row labeling is injective, distinct pairs
\((e,j)\) correspond to distinct entries of \(M\), and hence to distinct
hierarchical variables.  The determinant has total degree
\(D_r=(k-1)(r-1)\), and its integer coefficients have absolute value at
most \(D_r!\).  Applying Lemma~\ref{lem:dominance} with
$D=(k-1)(t-1)$, $H=D!$ and $L=Nk$,
we obtain $0<
\left|
\det\bigl((Q_o(G;p))_{\widehat e_0}\bigr)
\right|
<p$ for all sufficiently large primes \(p\).
\end{proof}

For an edge $e$, let $\operatorname{pos}(e)\in[N_r]$ be its position in
the chosen column order.  Define the signed cofactor vector
$\lambda=(\lambda_e)_{e\in E(G)}$ by
\[
 \lambda_e=\lambda_e(G;p)=(-1)^{\operatorname{pos}(e)-1}
                         \det\bigl((Q_o)_{\widehat e}\bigr).
\]
To see why this vector eliminates the tail coordinates, take any row
$\mathbf q$ of $Q_o$ and append it above $Q_o$. The resulting square
matrix has a repeated row, so its determinant is zero. Expanding along
the first row gives
$\sum_{e\in E(G)}q_e\lambda_e=0$.
Doing this for every row yields
\begin{equation}\label{eq:Qlambda}
 Q_o\lambda=0.
\end{equation}
By Lemma~\ref{lem:Q}, all $\lambda_e$ are nonzero both over $\mathbb Z$
and modulo $p$.  Every matrix entry in these cofactors comes from
columns $2,\ldots,k$ of $M$, and has absolute value at most $m_{N+1}$.
The determinant expansion therefore gives
\begin{equation}\label{eq:lambda-bound}
 1\le|\lambda_e|\le D_r!\,m_{N+1}^{D_r}=p^{o(1)}.
\end{equation}

\begin{lemma}[Positive normalization]\label{lem:normalization}
The edges of $G$ can be oriented, and the common sign of its cofactor
vector can be chosen, so that $\lambda_e>0$ for every edge $e$.
The resulting matrix and vector still satisfy $Q_o\lambda=0$,
and the absolute values of the cofactors are unchanged.
\end{lemma}

\begin{proof}
For signs $\tau_e\in\{1,-1\}$, reversing precisely those edges with
$\tau_e=-1$ replaces $Q_o$ by $Q_oD$, where
$D=\operatorname{diag}(\tau_e:e\in E(G))$ in the chosen column order.
Deleting column $e$ multiplies its determinant by
$\prod_{e'\ne e}\tau_{e'}=(\prod_{e'}\tau_{e'})\tau_e$.
Hence the cofactor vector of $Q_oD$ is
$(\prod_e\tau_e)D\lambda$.
Choose $\tau_e=\operatorname{sgn}(\lambda_e)$, where
$\operatorname{sgn}(x)=1$ for $x>0$ and $-1$ for $x<0$.
After changing one common sign, the new vector is
$D\lambda=(|\lambda_e|)_{e\in E(G)}$, and
$(Q_oD)(D\lambda)=Q_o\lambda=0$.
\end{proof}

From now on, $Q_o$ and $\lambda$ refer to this normalized orientation
and positive vector.  Thus $\lambda$ is the cofactor vector up to one
common sign.  For fixed initial data, the signs of the original
cofactors are constant for sufficiently large $p$, by the dominant
monomial argument in Lemma~\ref{lem:dominance}.

\begin{definition}[Edge weights and the head equation]\label{def:head-equation}
For an exact labeled graph $G$ with the choices above, define its
positive edge weights and divergence coefficients by
\[
 f_e(G;p)=M_{\rho(e),1}(p)\lambda_e(G;p),\qquad
 d_v(G;p)=\sum_{e\text{ leaves }v}f_e(G;p)
          -\sum_{e\text{ enters }v}f_e(G;p).
\]
Thus the divergence at a vertex is its total outgoing weight minus its
total incoming weight.  The integer \emph{head equation} of $G$ is
\[
 E_G(p):\quad \sum_{v\in V(G)}d_v(G;p)x_v=0.
\]
We write $f_e$ and $d_v$ when $G,p$ are understood, and put
$d(X)=\sum_{v\in X}d_v$ for $X\subseteq V(G)$.
\end{definition}

By \eqref{eq:lambda-bound},
\[
 0\le\log|\lambda_e|=O(\log m_{N+1})
                         =o(\log m_{\rho(e)}).
\]
Since $f_e=m_{\rho(e)}\lambda_e$, it follows that
\begin{equation}\label{eq:flow-order}
 \log f_e=(1+o(1))\log m_{\rho(e)},
 \qquad
 \rho(e)<\rho(e')\Longrightarrow \log f_{e'}=o(\log f_e).
\end{equation}

For fixed $t$ and $k$, there are only finitely many graphs on vertex sets $[r]$,
$2\le r\le t$, with injective edge labels in $[N]$, together with choices
of orientation and root.  Thus all estimates above have constants and
thresholds that can be chosen independently of these choices.

Every edge contributes once with each sign, so $\sum_vd_v=0$.
For $X\subseteq V(G)$, let $\delta_G(X)$ be the set of edges with
exactly one endpoint in $X$, and put $\sigma_X(e)=1$ if $e$ leaves $X$
and $\sigma_X(e)=-1$ if it enters $X$.  For every nonempty proper
$X\subsetneq V(G)$, cancellation of the internal edges gives
\begin{equation}\label{eq:cut-sum}
 d(X)=\sum_{e\in\delta_G(X)}\sigma_X(e)f_e\ne0.
\end{equation}
Indeed, the cut is nonempty because $G$ is connected. Let $e_*$ be its
edge of smallest row label. By \eqref{eq:flow-order},
\[
 \left|d(X)-\sigma_X(e_*)f_{e_*}\right|
 \le\sum_{e\in\delta_G(X)\setminus\{e_*\}}f_e=o(f_{e_*}),
\]
since there are at most $N$ edges. In particular, every $d_v$ is nonzero,
so $E_G(p)$ is invariant in the sense of Definition~\ref{def:invariant}.
The stronger fact that every proper nonempty vertex set has a nonzero
coefficient sum will also be used when we merge vertex blocks.

\section{From graph divergences to plastic permutation equations}
\label{sec:head}

Let $G$ be an exact labeled graph on $r$ vertices, where $3\le r\le t$.
We use its normalized weights and head equation from
Definition~\ref{def:head-equation}. We now convert $E_G(p)$ into an
equation of the form \eqref{eq:LU} generated by a plastic set.

\begin{definition}[Partial sums]\label{def:partial-sums}
Given an ordering $v_1,\ldots,v_r$ of the vertices, define
\[
 s_0=0,\qquad s_i=\sum_{j=1}^i d_{v_j}\quad(1\le i\le r).
\]

Notice that $s_r=s_0=0$.
We call $\{s_0,\ldots,s_{r-1}\}$ its
\emph{partial-sum set}.
\end{definition}

Equation~\eqref{eq:cut-sum} ensures that the $r$ elements are distinct: equality of two would give a zero coefficient sum
over a nonempty proper set of consecutive vertices.  The purpose of the
next lemma is to find an ordering for which these $s_0, \ldots, s_{r-1}$ also have the required separation of logarithmic sizes.

\begin{lemma}[Plastic partial sums]\label{lem:partial-sums}
Fix an integer-valued function $h=h(p)=p^{1-o(1)}$.
For every exact labeled graph $G$ on $r$ vertices, $3\le r\le t$,
the coefficients $(d_v)_{v\in V(G)}$
have a cyclic ordering whose partial-sum set, after a positive translation,
is $(\alpha_0,\beta_0)$-plastic with respect to $h$.
Here $0<\beta_0\le\alpha_0^t<1$ depend only on
$t,k,c$, where $c$ is the constant in Definition~\ref{def:hierarchy}.
The translated set is contained in $[h]$ for sufficiently large $p$.
The threshold here may also depend on the function $h$.
Consequently, after permuting its variables, $E_G(p)$ is $L(U)$ for a permutation $U$ of an
$(\alpha_0,\beta_0)$-plastic set.
\end{lemma}

\begin{proof}
We construct the ordering by first merging blocks of vertices and then
reversing those merges.  Start with the singleton partition of $V(G)$.
At each step choose the edge
of smallest row label whose endpoints lie in two different current blocks,
and merge those blocks.  Connectedness guarantees that the procedure ends
with one block.  The labels of the chosen edges strictly increase.

For a current block $X$, its coefficient is $d(X)$. This is nonzero whenever $X\ne V(G)$ by
\eqref{eq:cut-sum}.  Suppose that the next chosen edge $e$ joins blocks
$X$ and $Y$.  Every earlier edge is already internal to a current block and
therefore cancels from every block sum.  It follows from
\eqref{eq:flow-order} that, for some sign $\eps\in\{1,-1\}$ and real
remainders $r_X,r_Y$,
\[
       d(X)=\eps f_e+r_X,\qquad d(Y)=-\eps f_e+r_Y,
       \qquad |r_X|+|r_Y|=o(f_e),
\]
whereas all other current block coefficients are $o(f_e)$.  The two
coefficients have opposite signs, and merging replaces them by the nonzero
coefficient
\[
                       d(X\cup Y)=d(X)+d(Y),
\]
except at the final merge, when the sum is zero.

We now reverse the sequence of contractions and use it to construct the desired cyclic ordering of the coefficients.

In the last merging step, there are exactly two blocks, say \(X_1\) and \(Y_1\), joined by the last contraction edge \(e_1\). Since
$d(X_1)+d(Y_1)=\sum_{v\in V(G)}d_v=0$
and both coefficients are nonzero, we may assume without loss of generality that \(d(X_1)>0>d(Y_1)\). We order the two coefficients as
$(d(X_1),d(Y_1))=(d(X_1),-d(X_1))$,
and the corresponding partial-sum set is
\[
P_2=\{0,d(X_1)\}.
\]
Here and below, the subscript of \(P_i\) records the number of blocks in the current partition. Denote
\[
\eta_1=d(X_1).
\]
By the analysis of the forward contraction along \(e_1\),
$\eta_1=f_{e_1}+o(f_{e_1})$.

Next, we recover \(P_{i+1}\) recursively from \(P_i\) for \(2\le i\le r-1\). Suppose that the reverse process has reached a partition into \(i\) blocks, with partial-sum set \(P_i\). The next reverse step undoes a forward contraction along an edge \(e_i\): a current block \(X_i\cup Y_i\) is split into the two blocks \(X_i\) and \(Y_i\). Relabel \(X_i\) and \(Y_i\), if necessary, so that
$d(X_i)>0>d(Y_i)$.
Since
$d(X_i\cup Y_i)=d(X_i)+d(Y_i)$,
we replace the increment \(d(X_i\cup Y_i)\) in the current cyclic ordering by the two consecutive increments \(d(X_i)\) and \(d(Y_i)\), in this order. Let \(\xi_i\in P_i\) be the partial sum at which the increment \(d(X_i\cup Y_i)\) starts. The replacement is therefore
\[
\xi_i
\longrightarrow
\eta_i:=\xi_i+d(X_i)
\longrightarrow
\xi_i+d(X_i\cup Y_i).
\]
Thus all elements of \(P_i\) remain unchanged and exactly one new element \(\eta_i\) is inserted. Hence
\[
P_{i+1}=P_i\cup\{\eta_i\}.
\]

We next show that \(\eta_i\) is larger than every element of \(P_i\). Consider the corresponding stage of the forward contraction process immediately after \(e_i\) was contracted. Since \(e_i\) was the edge of smallest row label joining two distinct blocks just before that contraction, every edge with row label at most \(\rho(e_i)\) is then internal to one of the current blocks. Such edges cancel from the corresponding block coefficients. Therefore, for every current block \(Z\),
\[
|d(Z)|
\le
\sum_{\rho(e')>\rho(e_i)}f_{e'}
=o(f_{e_i}).
\]
Since the elements of \(P_i\) are partial sums of these \(i\) block coefficients,
\[
\operatorname{diam}(P_i)
\le
\sum_Z|d(Z)|
=o(f_{e_i}).
\]
As \(0\in P_i\), every element of \(P_i\), and in particular \(\xi_i\), has absolute value \(o(f_{e_i})\).

On the other hand, by the analysis of the forward contraction along \(e_i\),
\[
d(X_i)=f_{e_i}+o(f_{e_i}),
\qquad
d(Y_i)=-f_{e_i}+o(f_{e_i}).
\]
Hence
\[
\eta_i
=
\xi_i+d(X_i)
=
f_{e_i}+o(f_{e_i}).
\]
Since every element of \(P_i\) has absolute value \(o(f_{e_i})\), the new element \(\eta_i\) is the unique maximum of \(P_{i+1}\) for all sufficiently large \(p\).

Starting from \(P_2\) and applying this recursively for \(i=2,\ldots,r-1\), we obtain
\[
P_2\subset P_3\subset\cdots\subset P_r,
\]
where \(P_i\) consists of \(i\) distinct nonnegative elements and the unique element of \(P_{i+1}\setminus P_i\) is the new maximum of \(P_{i+1}\). In particular,
\[
P_r=\{0,\eta_1,\ldots,\eta_{r-1}\},
\qquad
0<\eta_1<\cdots<\eta_{r-1}.
\]

The forward contraction edges are chosen in strictly increasing order of their row labels, whereas \(e_1,e_2,\ldots,e_{r-1}\) are indexed in the reverse order. Thus
\[
\rho(e_1)>\rho(e_2)>\cdots>\rho(e_{r-1}).
\]
By \eqref{eq:flow-order} and the estimates above, for any $i\in [r-1]$,
$\tfrac{\log \eta_i}{\log m_{\rho(e_i)}}\to 1$.

We now translate \(P_r\) by
$\zeta=m_{N+1}(p)$.
Write the resulting set in increasing order as
$R=\{b_1<\cdots<b_r\}$,
where
\[
b_1=\zeta,
\qquad
b_{i+1}=\zeta+\eta_i
\quad (1\le i\le r-1).
\]
Since \(\rho(e_i)\le N\), the hierarchy implies that \(\zeta=o(\eta_i)\) for every \(i\). Consequently, for any $i\in [r-1]$,
$\tfrac{\log b_{i+1}}{\log m_{\rho(e_i)}}\to 1$.

Set
$\alpha_0=\frac{1+c}{2}$ and
$\beta_0=\frac12\min\{c^{N+1},\alpha_0^t\}$. Hence $\beta_0\le \alpha_0^t\le \alpha_0^r$.
Our next goal is to show that \(R\) is \((\alpha_0,\beta_0)\)-plastic with respect to \(h\).
First, \(b_r=p^{o(1)}\) and \(h=p^{1-o(1)}\), so
$R\subseteq[h]$.
Then,
$\log b_1=\log m_{N+1}(p)$
has logarithmic order \((\log p)^{c^{N+1}}\). As
$\tfrac{\log h}{\log p}\to1$
and \(\beta_0<c^{N+1}\), it follows that
\[
\log b_1
=
\Omega\bigl((\log h)^{\beta_0}\bigr).
\]
Also, as \(\rho(e_{r-1})\ge1\) and \(c<\alpha_0\),
\[
\log b_r
=
O\bigl((\log p)^c\bigr)
=
O\bigl((\log h)^{\alpha_0}\bigr).
\]
It remains to compare consecutive elements. Since \(\rho(e_1)\le N\),
$c^{N+1}
\le
c\,c^{\rho(e_1)}
<
\alpha_0c^{\rho(e_1)}$.
Hence
\[
\log b_1
=
O\bigl((\log b_2)^{\alpha_0}\bigr).
\]
For \(2\le i\le r-1\), we have
\(\rho(e_{i-1})>\rho(e_i)\), so
$c^{\rho(e_{i-1})}
\le
c\,c^{\rho(e_i)}
<
\alpha_0c^{\rho(e_i)}$.
It follows that
\[
\log b_i
=
O\bigl((\log b_{i+1})^{\alpha_0}\bigr)
\qquad (2\le i\le r-1).
\]
Thus \(R\) is \((\alpha_0,\beta_0)\)-plastic.

Finally, let \(v_1,\ldots,v_r\) be the cyclic ordering of the vertices obtained from the reverse process, retaining the starting point $0$. Let \(s_0,\ldots,s_r\) be the corresponding partial sums from Definition~\ref{def:partial-sums}. Then
\[
P_r=\{s_0,\ldots,s_{r-1}\},
\qquad
s_r=s_0=0.
\]
Define
\[
U=(s_1+\zeta,\ldots,s_r+\zeta).
\]
Since \(s_r=s_0\), \(U\) is a permutation of \(R\). Its cyclic differences satisfy
\[
s_1-s_r=d_{v_1},
\qquad
s_i-s_{i-1}=d_{v_i}
\quad (2\le i\le r).
\]
Therefore, after ordering the variables as
$x_{v_1},\ldots,x_{v_r}$,
$E_G(p)$ is exactly \(L(U)\).
\end{proof}

\section{The common head set}
\label{sec:common-head}

We now choose one set of head coordinates that avoids every equation from an exact collision graph.

\begin{definition}[The equation family]\label{def:equation-family}
Let $\mathcal G$ be the collection of exact labeled graphs on vertex
sets $[r]$, $3\le r\le t$, together with all choices of initial orientation
and root.  The vertex order is the natural order, and the row labels lie
in $[N]$.  For each $G\in\mathcal G$, use the normalization of
Lemma~\ref{lem:normalization} and the head equation of
Definition~\ref{def:head-equation}.  Define
\[
 \calE(p)=\{E_G(p):G\in\mathcal G\},\qquad
 B(p)=2\max_{G\in\mathcal G}\sum_{v\in V(G)}|d_v(G;p)|,
\]
and let
\[
 m_0=m_0(p)=\left\lfloor\frac{p-1}{B(p)}\right\rfloor
\]
be the length of the integer interval in which we seek a common
solution-free set.
\end{definition}

The indexing collection $\mathcal G$ is finite and depends only on $t,k$:
for each of the $N$ possible row labels, choose either no edge or one
of the finitely many vertex pairs.  It is also nonempty. On three
vertices, take edge multiplicities $k-1,k-1,1$ on the three pairs.
This is an exact graph with $2k-1\le N$ edges, so its edges can receive
distinct row labels. Thus $B(p)$ is positive for sufficiently large $p$.
By \eqref{eq:flow-order} and Definition~\ref{def:head-equation},
$B(p)=p^{o(1)}$, and hence $m_0=p^{1-o(1)}$ and
$\log m_0=(1+o(1))\log p$.

\begin{proposition}[The common head set]\label{prop:head-set}
There are constants $0<\gamma<1$, $K>0$ and $p_0$, depending only on
$t,k,c$, such that for every prime $p\ge p_0$ there is a set
$S=S(p)\subseteq\{0,\ldots,m_0(p)-1\}$ satisfying
\[
 |S|\ge m_0\exp\bigl(-K(\log m_0)^\gamma\bigr)=p^{1-o(1)}.
\]
Viewed as a subset of $\F_p$, the set $S$ is solution-free for every
equation in $\calE(p)$: if $x_v\in S$ for all $v\in V(G)$ and
$\sum_vd_v(G;p)x_v=0$ in $\F_p$, then all $x_v$ are equal.
\end{proposition}

\begin{proof}
For each $G\in\mathcal G$, apply Lemma~\ref{lem:partial-sums} with $h=m_0(p)$, and then apply Theorem~\ref{thm:plastic}.
This gives an integer $E_G(p)$-free set
$A_G\subseteq\{0,\ldots,m_0-1\}$ of size at least
$m_0\exp\bigl(-K_G(\log m_0)^{\gamma_G}\bigr)$ for some
$0<\gamma_G=\gamma(r,\alpha_0)<1$, $K_G>0$ and all primes above a
threshold $p_G$.  Here the constants also cover all $r!$ possible
permutations of the associated plastic set.  As $\mathcal G$ is a fixed
finite collection, we can take
\[
 \gamma=\max_{G\in\mathcal G}\gamma_G<1,\qquad
 K_0=\max_{G\in\mathcal G}K_G,\qquad
 p_0=\max_{G\in\mathcal G}p_G.
\]
After increasing $p_0$ so that $m_0\ge3$, we obtain
\[
 |A_G|\ge m_0\exp\bigl(-K_0(\log m_0)^\gamma\bigr)
 \qquad(G\in\mathcal G, p\ge p_0).
\]
Lemma~\ref{lem:common}, applied in this integer interval, supplies a
common solution-free set $S\subseteq\{0,\ldots,m_0-1\}$ with
\[
             |S|\ge \frac{m_0}{\exp(O((\log m_0)^\gamma))}
                  =p^{1-o(1)}.
\]
For every $G\in\mathcal G$ and all $x_v\in S$, the definition of $m_0$ gives
\[
 \left|\sum_vd_v(G;p)x_v\right|
 \le (m_0-1)\sum_v|d_v(G;p)|
 \le \frac{(m_0-1)B(p)}2<p.
\]
Thus congruence to zero modulo $p$ implies equality to zero over the
integers.  Viewing $S$ as a subset of $\F_p$, every solution to an equation in
$\calE(p)$ with coordinates in $S$ is therefore an integer solution,
so all its coordinates are equal. This proves the proposition.
\end{proof}

\section{The code and the proof of Theorem~\ref{thm:main}}
\label{sec:code}

For an arbitrary sufficiently large integer $q$, Bertrand's postulate
\cite[Chapter~2]{aigner2018proofs} gives a prime $p$ with $q/2<p\le q$.
Embed the $p$ symbols in an alphabet of size $q$.
Since $p^{k-o(1)}=q^{k-o(1)}$, it suffices to construct a $t$-PHF over
$\F_p$ of size $p^{k-o(1)}$.

Take $M(p)$ from Definition~\ref{def:matrix} and $S(p)$ from
Proposition~\ref{prop:head-set}. In this section, matrices and integer
coefficients are reduced modulo $p$ whenever they act on vectors over
$\F_p$; we retain the same notation for these reductions. Define
\begin{equation*}
 \calC_p=
 \left\{
 M\begin{pmatrix}x\\ \mathbf y\end{pmatrix}:
 x\in S,\ \mathbf y\in\F_p^{k-1}
 \right\}\subseteq\F_p^N.
\end{equation*}
Since $N-k=(t-2)(k-1)\ge0$, the matrix has at least $k$ rows.
By Corollary~\ref{cor:minors}, $M$ has column rank $k$, so the
parameterization is injective.  Consequently,
\begin{equation*}
                  |\calC_p|=|S|p^{k-1}=p^{k-o(1)}.
\end{equation*}

To prove Theorem~\ref{thm:main}, we only need to prove that $\calC_p$ is a $t$-PHF.

Suppose for a contradiction that $t$ distinct codewords are not separated by any of the $N$ coordinates.  In
each row choose one pair of codewords which agree in that row.  Regard the
chosen pair as an edge labeled by that row.  This gives a loopless
multigraph $H$ on $[t]$ with $N=N_t$ distinctly labeled edges.

Notice that no pair of vertices supports $k$ edges. Indeed, equality of the corresponding two codewords in $k$ rows would put the difference of their parameter vectors in the kernel of a $k\times k$ row submatrix of $M$.
Its invertibility, by Corollary~\ref{cor:minors}, forces the two parameter vectors to be equal and hence the two codewords are the same.
Choose a vertex set $W\subseteq[t]$ of minimum cardinality among all choices with $|W|\ge2$ and
\[
                  e_H(W)\ge N_{|W|}.
\]
Such $W$ exists, as the whole vertex set $[t]$ satisfies this inequality.
Moreover, since $N_2=k$ and no pair supports $k$ edges, we have $|W|>2$.
Put $r=|W|\ge3$, and select exactly $N_r$ edges from $H[W]$.
The resulting labeled graph $G$ is exact: by the minimality of $W$,
every proper subset $X\subsetneq W$ with $|X|\ge2$ satisfies
$e_G(X)\le e_H(X)<N_{|X|}$, and therefore
$e_G(X)\le(k-1)(|X|-1)$. For a singleton, the same inequality holds
because there are no loops.

Write the parameter vector of the codeword at \(v\in W\) as
$\mathbf z_v=(x_v,z_{v,2},\ldots,z_{v,k})^{\mathsf T}$, where \(x_v\in S\) and
$\mathbf y_v=(z_{v,2},\ldots,z_{v,k})^{\mathsf T}\in\F_p^{k-1}$. Choose a root \(o\in W\) and identify \(W\)
with \([r]\). Give the edges an initial orientation and then normalize it
as in Lemma~\ref{lem:normalization}. This determines an equation \(E_G(p)\) in
the family of Definition~\ref{def:equation-family}.

If \(e\) is directed from \(u\) to \(v\) and has row label \(\rho(e)\),
then the collision corresponding to \(e\) is
$(M\,\mathbf z_u)_{\rho(e)}= (M\,\mathbf z_v)_{\rho(e)}$, and hence we have the following equation:
\begin{equation*}
 \sum_{j=1}^k M_{\rho(e),j}
       (z_{u,j}-z_{v,j})=0,
\end{equation*}
where $z_{v,1}=x_v$. Multiply the equation for each edge $e$ by
$\lambda_e$ and sum over all edges. We will show that the resulting
equation is $E_G(p)$.

For $v\in W$ and $j\in[k]$, let $c_{v,j}$ denote the coefficient of
$z_{v,j}$ in this weighted sum. We first show that all tail coefficients,
those with $j\ge2$, vanish.
On the one hand, for every non-root vertex \(w\ne o\) and every
\(j\in\{2,\ldots,k\}\), by \eqref{eq:Qlambda} we have
\[
c_{w,j}
=
(Q_o\lambda)_{(w,j)}
=
0.
\]
On the other hand, for each fixed \(j\in\{2,\ldots,k\}\), we claim that
\(c_{o,j}=0\). Indeed, for any edge \(e\) from $u$ to $v$, the contribution of
the weighted collision equation for $e$ to the coefficients of the \(j\)-th
coordinates is
$+M_{\rho(e),j}\lambda_e$ at $u$, and
$-M_{\rho(e),j}\lambda_e$ at $v$.
Hence the total contribution of \(e\) to
\(\sum_{v\in W}c_{v,j}\) is zero. Summing over all edges gives
\[
\sum_{v\in W}c_{v,j}=0.
\]
Since \(c_{v,j}=0\) for every \(v\ne o\), it follows that
$c_{o,j}
=
-\sum_{v\in W\setminus\{o\}}c_{v,j}
=
0.$
To sum up,
\[
c_{v,j}=0
\qquad
(v\in W,\ 2\le j\le k).
\]

It remains to compute the head coefficients $c_{v, 1}$. For an edge $e$ from $u$ to $v$,
\[
\lambda_e M_{\rho(e),1}(x_u-x_v)
=
f_e(x_u-x_v).
\]
Hence, for every \(v\in W\),
$c_{v,1}
=
\sum_{e\text{ leaves }v} f_e
-
\sum_{e\text{ enters }v} f_e
=
d_v$.
Thus the weighted sum is precisely $E_G(p)$ from
Definition~\ref{def:head-equation}.
Since \(x_v\in S\) for every \(v\in W\), Proposition~\ref{prop:head-set} gives
\[
x_v=x_o
\qquad
\text{for every }v\in W.
\]

We now return to the individual collision equations and determine
the tail coordinates $\mathbf y_v$. Let \(Y\) be the column vector indexed in the row order of \(Q_o\), with
\[
Y_{(v,j)}=z_{v,j}-z_{o,j},
\qquad
v\in W\setminus\{o\},\quad 2\le j\le k.
\]
Set $Y_{(o,j)}=0$ for $2\le j\le k$. Since all $x_v$ are equal,
the $j=1$ term of the collision equation for any edge $e$ from $u$ to $v$
vanishes, and the equation becomes
\[
\sum_{j=2}^k
M_{\rho(e),j}
\bigl(Y_{(u,j)}-Y_{(v,j)}\bigr)
=
0.
\]
By Definition~\ref{def:tail-matrix} and \eqref{eq:Q0Y}, after collecting over all \(e\in E(G)\), we get
\[
Q_o^{\mathsf T}Y=\mathbf 0.
\]
By Lemma~\ref{lem:Q}, \(Q_o\) has full row rank over \(\F_p\). Therefore
$Y=\mathbf 0$.
It follows that
$z_{v,j}=z_{o,j}$ for any
$v\in W,\ 2\le j\le k$,
or equivalently, for any $v\in W$,
\[
(x_v,\mathbf y_v)=(x_o,\mathbf y_o).
\]
Thus all parameter vectors indexed by \(W\) are identical, and hence so are the corresponding codewords, contradicting the choice of distinct codewords.

Therefore \(\calC_p\) is a \(t\)-PHF, and we have finished the proof of  Theorem~\ref{thm:main}.

\medskip
\noindent\textbf{Acknowledgements and history of the work.}
Work on this paper began during the preparation of \cite{WEI2025106075}.
We thank Chong Shangguan for suggesting that we pursue this problem using the construction of Alon and Shapira~\cite{alon2006extremal}.
We also thank Hong Liu for helpful discussions and comments on the manuscript.
Partial results were obtained during Xin Wei's visit to the National University of Singapore. 

\medskip
\noindent\textbf{Declaration on the use of AI.}
ChatGPT was used to assist in developing the proof of Lemma~\ref{lem:Q}, as well as preparing the manuscript.
All AI-suggested arguments were independently verified, revised and streamlined by the authors, who take full responsibility for the correctness of all mathematical results presented in this paper.

\bibliographystyle{abbrv}
\bibliography{RCF}

\end{document}